\documentclass[11pt]{amsart}

\usepackage{amsfonts, amsmath, amscd}
\usepackage[psamsfonts]{amssymb}

\usepackage{amssymb}

\usepackage[usenames]{color}

\newtheorem{theorem}{Theorem}[section]

\newtheorem{question}[theorem]{Question}

\newtheorem{proposition}[theorem]{Proposition}

\newtheorem{fact}[theorem]{Fact}

\numberwithin{equation}{section}

\newcommand{\CC}{C_k}
\newcommand{\NN}{\mathbb{N}}

\newcommand{\e}{\varepsilon}

\newcommand{\Tt}{\mathcal{T}}

\newcommand{\IR}{\mathbb{R}}

\newcommand{\ZZ}{\mathbb{Z}}

\newcommand{\Ss}{\mathbb{S}}

\newcommand{\Ff}{\mathfrak{F}}
\newcommand{\ZN}{\ZZ^{(\NN)}}

\renewcommand{\phi}{\varphi}

\input xy
\xyoption{all}

\title[A  characterization of barrelledness of $C_p(X)$]{A  characterization of barrelledness of $C_p(X)$}

\author{S.~Gabriyelyan}
\address{Department of Mathematics, Ben-Gurion University of the
Negev, Beer-Sheva, P.O. 653, Israel}
\email{saak@math.bgu.ac.il}

\begin{document}

\begin{abstract}
We prove that, for a Tychonoff space $X$, the space $C_p(X)$ is barrelled if and only if it is a Mackey group.
\end{abstract}

\maketitle

\section{Introduction}

For a Tychonoff space $X$ we denote by $C_p(X)$ the space $C(X)$ of all continuous functions on $X$ endowed with the pointwise topology. The relation between locally convex properties of $C_p(X)$  and topological properties of $X$ is illustrated by the following famous result  (all relevant definitions see Section \ref{sec:Mackey-Cp}).
\begin{theorem}[Buchwalter--Schmets] \label{t:Mackey-barrelled}
For a Tychonoff space $X$, the space $C_p(X)$ is barrelled if and only if every functionally bounded subset of $X$ is finite.
\end{theorem}
In this paper we give another characterization of barrelledness of $C_p(X)$ using the notion of Mackey group in the class $\mathcal{LQC}$ of all locally quasi-convex abelian groups.

For an abelian topological group $G$ we denote by $\widehat{G}$ the group of all continuous characters of $G$. Two topologies  $\tau$ and $\nu$ on an abelian group $G$  are said to be {\em compatible } if $\widehat{(G,\tau)}=\widehat{(G,\nu)}$. Being motivated by the classical Mackey--Arens theorem the following notion was  introduced and studied in \cite{CMPT} (see also \cite{NietoMP,Gab-Mackey}): a locally quasi-convex abelian group $(G,\tau)$ is called a {\em Mackey group in  $\mathcal{LQC}$} or simply a {\em Mackey group} if for every compatible locally quasi-convex group topology $\nu$ on $G$ associated with $\tau$ it follows that $\nu\leq \tau$. Every barrelled locally convex space (lcs for short)  is a Mackey group by \cite{CMPT}, see also Proposition \ref{p:Mackey-Chasco} below. If a lcs $E$ is a Mackey group   then it is also a Mackey space,  but the converse is not true in general. The first example of a Mackey (even metrizable) lcs  $E$ which is not a Mackey group is given in \cite{Gab-Mackey}. For any Tychonoff space $X$ the space $C_p(X)$ is a Mackey {\em space} because it is  quasibarrelled by Corollary 11.7.3 of \cite{Jar}. However, it turns out that to be a Mackey {\em group} for $C_p(X)$ is equivalent to be a barrelled space.
\begin{theorem} \label{t:Mackey-barrelled-main}
For a Tychonoff space $X$, the space $C_p(X)$ is barrelled if and only if it is a Mackey group.
\end{theorem}


\section{Proof of Theorem \ref{t:Mackey-barrelled-main}} \label{sec:Mackey-Cp}

Recall that a subset $A$ of a topological space $X$ is called {\em functionally bounded in $X$} if every continuous real-valued function on $X$ is bounded on $A$.

In what follows we need some notations. Denote by $\mathbb{S}$ the unit circle group and set $\Ss_+ :=\{z\in  \Ss:\ {\rm Re}(z)\geq 0\}$.
 If $\chi\in \widehat{G}$, it is considered as a homomorphism from $G$ into $\mathbb{S}$. A subset $A$ of $G$ is called {\em quasi-convex} if for every $g\in G\setminus A$ there exists   $\chi\in \widehat{G}$ such that $\chi(g)\notin \Ss_+$ and $\chi(A)\subseteq \Ss_+$.
If $A\subseteq G$ and $B\subseteq \widehat{G}$ set
\[
A^\triangleright :=\{ \chi\in \widehat{G}: \chi(A)\subset \Ss_+\}, \quad B^\triangleleft:=\{ g\in G: \chi(g)\in\Ss_+ \; \forall \chi\in B\}.
\]
Then $A$ is quasi-convex if and only if $A^{\triangleright\triangleleft}=A$.
An abelian topological group is called {\em locally quasi-convex} if it admits a neighborhood base at zero consisting of quasi-convex sets.

Let $G$ be an abelian topological group such that $\widehat{G}$ separates the points of $G$. Denote by $\sigma(G,\widehat{G})$ the {\em weak topology} on $G$, i.e., the smallest group topology on $G$ for which the elements of $\widehat{G}$ are continuous. In the dual group $\widehat{G}$, we denote by $\sigma(\widehat{G},G)$ the topology of pointwise convergence.  Proposition 1.5 of \cite{Ban} implies
\begin{fact} \label{f:Mackey-Ban}
If $U$ is a neighborhood of zero of an abelian topological group $G$, then $U^\triangleright$ is a $\sigma(\widehat{G},G)$-compact subset of $\widehat{G}$.
\end{fact}

Let $E$ be a nontrivial locally convex space and denote by $E'$ the topological dual space of $E$. Clearly,  $E$ is also an abelian topological group. Therefore we can consider the group $\widehat{E}$ of all continuous characters of $E$. The next important result is proved in \cite{HeZu,Smith}, see also \cite[23.32]{HR1}.
\begin{fact} \label{f:Mackey-group-lcs}
Let $E$ be a locally convex space. Then the mapping $p:E' \to \widehat{E}$, defined by the equality
\[
p(f)=\exp\{ 2\pi i f\}, \quad \mbox{ for all } f\in E',
\]
is a group isomorphism between $E'$ and $\widehat{E}$.
\end{fact}

The following fact is Lemma 1.2 of \cite{ReT} (for a more general result with a simpler proof, see \cite{Ga-Top}).
\begin{fact} \label{f:Mackey-ReT}
Let $E$ be a locally convex space. Then the space $(E, \sigma(E,E'))$ and the group $(E, \sigma(E,\widehat{E}))$ have the same family of compact subsets. Consequently, if $K$ is a $\sigma(\widehat{E},E)$-compact subset of $\widehat{E}$, then the subset $p^{-1}(K)$ of $E'$ is   $\sigma(E',E)$-compact, where $p$ is defined in Fact \ref{f:Mackey-group-lcs}.
\end{fact}

Recall that a lcs $E$ is {\em barrelled} if and only if every $\sigma(E',E)$-bounded  subset of $E'$ is equicontinuous. By Proposition 5.4 of \cite{CMPT}, every barrelled space $E$ is a Mackey group. For the sake of completeness and the convenience of the reader, we give below a direct and short proof of this fact.
\begin{proposition}[\cite{CMPT}] \label{p:Mackey-Chasco}
Every barrelled space $(E,\tau)$ is a Mackey group.
\end{proposition}
\begin{proof}
Let $\nu$ be a locally quasi-convex compatible group topology on $E$ associated with $\tau$ and let $U$ be a quasi-convex $\nu$-neighborhood of zero. We have to show that $U$ is also a $\tau$-neighborhood of zero.

Fact \ref{f:Mackey-Ban} implies that $U^\triangleright$ is a compact subset of $(\widehat{E}, \sigma(\widehat{E},E))$. Hence the set $K:=p^{-1}(U^\triangleright)$ is a $\sigma(E',E)$-compact subset of $E'$ by Fact \ref{f:Mackey-ReT}. As $(E,\tau)$ is barrelled,  $K$ is equicontinuous. Hence the polar $K^\circ$ of $K$ in the lcs $E$ is a $\tau$-neighborhood of zero, see \cite[Theorem 8.6.4]{NaB}. Therefore
\[
\begin{split}
\frac{1}{4} K^\circ & =\left\{ x\in E: |\chi(x)|\leq \frac{1}{4}, \forall \chi\in K\right\} \\
& \subseteq \{ x\in E: p(\chi)(x)=\exp\{ 2\pi i \chi(x)\} \in\Ss_+ , \forall\chi\in K\} =\big( p(K)\big)^\triangleleft = U^{\triangleright\triangleleft} =U.
\end{split}
\]
Thus $U$ is a $\tau$-neighborhood of zero.
\end{proof}

Recall that, for a Tychonoff space $X$, the sets of the form
\[
W[y_1,\dots,y_s,\e]:= \{ f\in C(X): |f(y_i)|<\e, i=1,\dots,s\}
\]
form a base at zero in $C_p(X)$, and  the dual space of $C_p(X)$ is the space
\begin{equation} \label{equ:Mackey-dual-Cp}
(C_p(X))'=\left\{ \sum_{i=1}^m \lambda_i \delta_{x_i}: x_1,\dots,x_m \in X, \; \lambda_1,\dots,\lambda_m \in\IR, \; m\in\NN\right\},
\end{equation}
where $\delta_x (f):= f(x)$.

The following result is Lemma 11.7.1 of \cite{Jar}.
\begin{fact} \label{l:Mackey-discrete}
If $A$ is an infinite subset of a Tychonoff space $X$, then  there is a one-to-one sequence $\{ a_n \}_{n\in\NN}$ in $A$ and a sequence $\{ U_n\}_{n\in\NN}$  of open sets in $X$ such that $a_n\in U_n$ and $\overline{U_n}\cap \overline{U_k}=\emptyset$ for every distinct $n,k\in\NN$.
\end{fact}

The following group plays an essential role in the proof of Theorem \ref{t:Mackey-barrelled-main}.
Set
\[
c_0 (\Ss):= \{ (z_n) \in \Ss^\mathbb{N} :\; z_n \to 1\},
\]
and denote by $\Ff_0 (\Ss)$ the group $c_0 (\Ss)$ endowed with the metric  $d\big((z_n^1), (z_n^2 ) \big)= \sup \{ |z_n^1 -z_n^2 |, n\in\NN \}$. Then $\Ff_0 (\Ss)$ is a Polish group, and the sets of the form $V^\NN \cap c_0(\Ss)$, where $V$ is an open neighborhood at the identity of $\Ss$, form a base at the identity in $\Ff_0 (\Ss)$. Actually $\Ff_0 (\Ss)$ is isomorphic to $c_0/\ZN$  (see \cite{Gab}), where $\ZN$ is the direct sum $\bigoplus_\mathbb{N} \mathbb{Z}$. Denote by $Q:c_0\to \Ff_0 (\Ss)$ the quotient map, so $Q\big( (x_n)_{n\in\NN}\big) = \big( q(x_n)\big)_{n\in\NN}$, where $q:\IR\to \Ss$ is defined by $q(x)=\exp\{ 2\pi i x\}$.
The group $\Ff_0 (\Ss)$ is locally quasi-convex and $\widehat{\Ff_0 (\Ss)}=\ZN$ by Theorem 1 of \cite{Gab}.

For $x\in\IR$ we write $[x]$ for the integral part of  $x$.
Now we are ready to prove Theorem \ref{t:Mackey-barrelled-main}.

\begin{proof}[Proof of Theorem \ref{t:Mackey-barrelled-main}]
By the Buchwalter--Schmets theorem and Proposition \ref{p:Mackey-Chasco} to prove the theorem we have to show only that, if $C_p(X)$ is a Mackey group, then every functionally bounded subset of $X$ is finite. Suppose for a contradiction that there is an infinite functionally bounded subset $A$ of $X$. By Fact \ref{l:Mackey-discrete}, there is a one-to-one sequence $\{ a_n \}_{n\in\NN}$ in $A$ and a sequence $\{ U_n\}_{n\in\NN}$  of open sets in $X$ such that $a_n\in U_n$ and $\overline{U_n}\cap \overline{U_k}=\emptyset$ for every distinct $n,k\in\NN$.

For every $n\in\NN$ set $\chi_n :=\frac{1}{n} \delta_{a_n}$. As $A$ is functionally bounded in $X$,  $\chi_n(f)\to 0$ for every $f\in C(X)$. So we can define the linear injective operator $F: C(X) \to C_p(X)\times c_0$ and monomorphism $F_0 : C(X) \to C_p(X)\times \Ff_0(\Ss)$  setting ($\forall f\in C(X)$)
\[
\begin{split}
F(f) & := \big( f, R(f)\big), \mbox{ where } R(f):= \big(\chi_n(f)\big) \in c_0, \\
F_0(f)& := \big( f, R_0(f)\big), \mbox{ where } R_0(f):= Q\circ R(f)=\big( \exp\{ 2\pi i \chi_n(f)\} \big) \in \Ff_0(\Ss).
\end{split}
\]
Denote by $\Tt$ and $\Tt_0$ the topology on $C(X)$ induced from $C_p(X)\times c_0$ and $C_p(X)\times \Ff_0(\Ss)$, respectively. So $\Tt$ is a locally convex vector topology on $C(X)$ and $\Tt_0$ is a locally quasi-convex group topogy on $C(X)$ (since the group $\Ff_0(\Ss)$ is locally quasi-convex, and a subgroup of a product of locally quasi-convex groups is clearly locally quasi-convex). Denote by $\tau_p$ the pointwise topology on $C(X)$. Then, by construction, $\tau_p \leq \Tt_0 \leq\Tt$, so taking into account Fact \ref{f:Mackey-group-lcs} we obtain
\begin{equation}\label{equ:Mackey-Cp-1}
\widehat{ C_p(X)} \subseteq \widehat{\big( C(X),\Tt_0\big)} \subseteq p\big(  (C(X),\Tt)' \big).
\end{equation}

Let us show that the topologies $\tau_p$ and $\Tt_0$ are compatible. By (\ref{equ:Mackey-Cp-1}), it is enough to show that each character of $\big( C(X),\Tt_0\big)$ is a character of $C_p(X)$. Fix $\chi\in \widehat{\big( C(X),\Tt_0\big)}$. Then  (\ref{equ:Mackey-Cp-1}) and the Hahn--Banach extension theorem imply that $\chi=p(\eta)$ for some
\[
\eta =\big(\xi,(c_n)\big)\in (C_p(X))' \times \ell_1, \mbox{ where } \xi \in (C_p(X))' \mbox{ and } (c_n)\in \ell_1,
\]
and
\[
\eta (f)=\xi(f) + \sum_{n\in\NN} c_n \chi_n (f), \quad \forall f\in C(X).
\]
To prove that $\chi\in \widehat{ C_p(X)}$ it is enough to show that $c_n=0$ for almost all indices $n$. Suppose for a contradiction that  $|c_n|>0$ for infinitely many indices $n$. Take a neighborhood $U$ of zero in $\Tt_0$ such that (see Fact \ref{f:Mackey-group-lcs})
\begin{equation} \label{equ:Mackey-Cp-2}
\eta(U) \subseteq \left( -\frac{1}{10},\frac{1}{10}\right) + \mathbb{Z}.
\end{equation}
We can assume that $U$ has a canonical form
\[
U=F_0^{-1}\left( \left( W[y_1,\dots,y_s;\e]\times \big(V^\NN \cap c_0(\Ss)\big)\right) \cap F_0\big(C(X)\big) \right),
\]
for some points $y_1,\dots,y_s\in X$, $\e >0$ and a neighborhood $V$ of the identity of $\Ss$.
Let $\xi = \sum_{i=1}^m \lambda_i \delta_{x_i}$ for some points $x_1,\dots, x_m\in X$ and $\lambda_1,\dots,\lambda_m \in\IR$, see (\ref{equ:Mackey-dual-Cp}). Since $|c_n|>0$ for infinitely many indices and all $a_n$ are distinct, we can find an index $\alpha$ such that $0<|c_\alpha|<1/100$ (recall that $(c_n)\in\ell_1$) and $a_\alpha \not\in \{ x_1,\dots, x_m, y_1,\dots,y_s\}$. Choose a continuous function $f:X \to [0,1]$ such that
\[
f(X\setminus U_\alpha)=\{ 0\} , \; f(x_1)=\dots = f(x_m)=f(y_1)=\dots =f(y_s)=0 \mbox{ and }  f(a_\alpha)=1,
\]
and set $h(x)= \left[ \frac{1}{4c_\alpha} \right] f(x)$. As $R_0(h)=0$, it is clear that $h\in U$. Setting $r_\alpha := \frac{1}{4c_\alpha} - \left[ \frac{1}{4c_\alpha}\right]$ (and hence $0\leq r_\alpha <1$) we obtain
\[
\frac{1}{4}-\frac{1}{100} < \eta (h)= c_\alpha h(a_\alpha) =c_\alpha \left( \frac{1}{4c_\alpha} - r_\alpha\right) = \frac{1}{4} -c_\alpha r_\alpha <\frac{1}{4}+\frac{1}{100}.
\]
But these inequalities contradict the inclusion (\ref{equ:Mackey-Cp-2}). This contradiction shows that $c_n=0$ for almost all indices $n$, and hence $\eta\in \big(C_p(X)\big)'$. Thus $\tau_p$ and $\Tt_0$ are compatible.

To finish the proof we have to show that the topology  $\Tt_0$  is strictly finer than $\tau_p$. For every $n\in\NN$ take a function $f_n:X\to [0,1/2]$ such that
\[
f_n(X\setminus U_n)=\{ 0\}  \mbox{ and }  f(a_n)=1/2.
\]
As $\overline{U_n}\cap \overline{U_k}=\emptyset$ for every distinct $n,k\in\NN$, $f_n \to 0$ in the pointwise topology $\tau_p$. On the other hand,
\[
F_0(f_n) =\big( f_n, (0,\dots,0,-1,0,\dots)\big),
\]
where $-1$ is placed in position $n$. So $f_n\not\to 0$ in $\Tt_0$. Thus $\Tt_0$ is strictly finer than $\tau_p$.
\end{proof}

We end with the following question.
\begin{question}
For which Tychonoff spaces $X$ the space $\CC(X)$ of all continuous functions endowed with the compact-open topology is a Mackey group?
\end{question}

\bibliographystyle{amsplain}

\end{document}